\documentclass{article}
\usepackage{graphicx}
\usepackage{subfigure}
\usepackage{amsmath}
\usepackage{amssymb}
\usepackage{amsthm}
\usepackage{float}
\usepackage[top=2cm, bottom=2cm, left=2cm, right=2cm]{geometry}
\newtheorem{example}{Example}
\newtheorem{theorem}{Theorem}

\usepackage{authblk}

\usepackage[compress]{cite}
\usepackage{xcolor}
\usepackage{enumitem}
\usepackage{epstopdf}

\begin{document}

	\title{Modular chaos   for random processes}	
	
	\author[]{Marat Akhmet\thanks{Corresponding Author, Tel.: +90 312 210 5355, Fax: +90 312 210 2972, E-mail: marat@metu.edu.tr} }

	\affil[]{Department of Mathematics, Middle East Technical University, 06800 Ankara, Turkey}

	\date{}
	\maketitle

	\noindent \textbf{Abstract.}  
		In the present paper, an essential  generalization of the  symbolic dynamics   is considered.  We apply  the  notions of abstract self-similar sets and the  similarity map  for  a   chaos introduction,  which orbits are expanded among  infinitely  many  modules.  The dynamics is  free of dimensional,  metrical  and topological  assumptions.   It    unites  all the three types of Poincar\'{e}, Li-Yorke and Devaney  chaos  in a single model, which  can  be unbounded.  The   research  demonstrates  that  the dynamics   of Poincar\'{e} chaos is of exceptional  use   to analyze discrete and continuous-time random processes.   Examples,  illustrating the  results  are provided.

	\section{Introduction} 
	\label{sec1}
	
	The paradigm  of chaos relies  on the two  distance antagonists,   proximality   and separation.    They  utilize  metrical  or topological  closeness and    divergence in the description of chaos considering  either couples of  orbits \cite{Devaney,Yorke} or a single motion \cite{AkhmetPoincare}.    In the present research,    we  have  renounced the idea that transitivity  and density of  a class of motions have  to be described metrically, and   even topologically.   
	 Our  suggestions are  based on the  new idea that  closeness is  considered  not   in the  metrical  sense,  but    through \textit{ indexing},  although   separation still rely  on the distance.  That  is,  the antagonists still  are  present,  but  not   in the  way of  conventional theories.      The process of indexing    is adjacent   with  the similarity  map.  	More precisely,  the indexing  is  an  instrument  how the  mapping is to be chaotic.
	   Moreover,  trajectories  move   through   different sets,  and even various metric spaces,   when  the  time increases,    and  each  of the  sets  admits its own chaotic structure.  This is  why,   we say  that  the \textit{modular chaos}  is in the  focus of the  paper or more  precisely \textit{modular domain-structured chaos}. 
	 
	 	In  the paper\cite{AkhmetDiscreteRandom} the  notion  of domain-structured chaos is  presented in most abstract  form,    free of topological  properties for  domains.  This  time,  we intend  to  consider the  chaos developed to  the  level, when  two  points are  close, if they  belong to  the  same set by  \textit{indexing}.

	 	 It is   meaningful  that  the approach helps us to  find  a chaotic structure  for  continuous-time random processes, bounded    as well as  unbounded.  We accept  the phenomenon as chaos presence in the dynamics,  since  realizations of the corresponding  discrete  time random process with  fixed discrete  sequence of time  are the same as trajectories of  the  deterministic  modular  similarity  map introduced in the present  research.    One must say  that the concepts  of Poincar\'{e} chaos  and unpredictable sequences are of strong  importance  for the  investigation.  
	 	 
	 	 We hope that  the concept   will be  developed further,  by  considering multidimensional indexing and infinitely  dimensional dynamics.    Additionally,   the  definitions can  be   considered for  processes,  which  are stochastic only  partially, in  time,  that  is hybrid dynamics, through combinations of intervals,  where  a dynamics is deterministic and/or  stochastic.  The suggestions can be considered  for construction of hierarchy of chaos such  that  each next  level  of the  complex dynamics can  be determined as union of previous  ones.

	     In the  basis of our  investigation are  the  most  effective  methods, which have been  developed    for   chaos  investigation   through  \textit{similarity,  self-similaruty}   and \textit{symbolic dynamics}. 
		
		 The  research  on similarity goes back to  G. Leibniz, who introduced the notions of recursive self-similarity \cite{Zmeskal}.  The idea of a self-similar set was first considered by Moran,\cite{Moran}.  He  gave a mathematical definition of a geometric construction as a collection of sets satisfying specific conditions.  Further,   definitions of self-similarity and related problems were  discussed in papers \cite{Jorgensen, Stella,PesinWeiss,Hutchinson,Hata,Edgar,Spear,Bandt,Falconer1,Lau,Ngai}. 
		  Our research is  based on  a point-set structure in metric spaces and this is  why we call it \textit{abstract self-similarity}\cite{AkhmetSimilarity}. It  does not rely on any special functions, and the \textit{similarity map}  is naturally  combined  with  the  concept.  The map can be considered as a generalization of the Bernoulli shift on symbolic spaces \cite{Wiggins88}.   We suppose that next extension of the  present results  can be  obtained  on the basis of papers   \cite{Hutchinson,Hata,Edgar,Spear,Bandt,Falconer1,Lau,Ngai}.

		  Potential   applications of the  method  naturally   relate  to ways  of  chaos research such as chaos control and synchronization \cite{Gonzales}.   The  theory  of   stochastic differential  equations   \cite{Oksendal} can  be enriched,     if modular  chaos investigation  results  are involved in  the  study  by utilizing  methods of replication of chaos and unpredictable dynamics \cite{AkhmetBook,AkhmetReplication,AkhmetPoincare}. 
		
		\section{Modular domain-structured chaos} 
		
  	Let    a  countable collection of metric spaces $(F_j,d_j),    j=1,2,\ldots,$   be given,   with    distances $d_j.$   Assume   that  for    each  set $F_j$   the    following presentation exists,  
			\begin{equation} \label{AbstFracSet}
		  \mathcal{F}^j =  \big\{\mathcal {F}^j_{i_1 i_2 ... i_n ... } : i_k=1,2, ..., m, \; k \big\}, j=1, 2, ...,
		\end{equation}
		where  $m$  is a natural number   not  smaller than two and common for all these sets.  
		It  means that  each   element  of the  set  $ F_j$  is   labeled  through at  least  by  one  member of   $\mathcal {F}^j,$   and each   element  of  set    $\mathcal {F}^j$  presents a  member from $F_j.$       We   assume that  the   uniqueness  is not  necessarily  required  for the relation.   That  is, if 
		$\mathcal{F}^j_{i_1 i_2 ... i_n \ldots}$  and $\mathcal{F}^j_{j_1 j_2 ... jn, \ldots}$   present the same   element  $f \in \mathcal F_j,$   it  is not  necessary  that   $i_n=j_n$ for all $n =1,2,\ldots.$    Moreover,    we    determine    functions 
		$\delta_j(\mathcal{F}^j_{i_1 i_2 ... i_n \ldots}, \mathcal{F}^j_{j_1 j_2 ... jn, \ldots} ) =  d_j(f_1,f_2), j =1,2,\ldots,$  if  	$\mathcal{F}^j_{i_1 i_2 ... i_n \ldots}$  and $\mathcal{F}^j_{j_1 j_2 ... jn, \ldots}$   correspond to elements $f_1$ and $f_2$  of the set  $F_j$ respectively,   such  that  
		$\delta_j(\mathcal{F}^j_{i_1 i_2 ... i_n \ldots}, \mathcal{F}^j_{j_1 j_2 ... jn, \ldots} ) = 0$  for   different  presentations of the same point in $F_j.$ 		
		 			 
		The following  sets are  needed, 
		\begin{equation} \label{AbstFracSubSet}
		\mathcal{F}^j_{i_1 i_2 ... i_n} = \bigcup_{j_k=1,2, ..., m } \mathcal{F}^j_{i_1 i_2 ... i_n j_1 j_2 ... },
		\end{equation}
		where indices  $ i_1, i_2, ..., i_n,$  are fixed. 	It is clear that
		\[ \mathcal{F}^j \supseteq \mathcal{F}^j_{i_1} \supseteq \mathcal{F}^j_{i_1 i_2} \supseteq ... \supseteq \mathcal{F}^j_{i_1 i_2 ... i_n} \supseteq \mathcal{F}^j_{i_1 i_2 ... i_n i_{n+1}} ... , \; i_k=1, 2, ... , m, \; k=1, 2, ... \, .\]
		That is, the sets form a nested sequence.   
		
		We will  say   that  the \textit{diameter condition} is valid for  sets $ \mathcal{F}^j_{i_1 i_2 ... i_n}, j=1,\ldots,$  if   
		\begin{equation} \label{Diamprop}
		\mathrm{max}_{i_1 i_2 ... i_n}\mathrm{diam}(\mathcal{F}^j_{i_1 i_2 ... i_n}) \to 0 \;\; \text{as} \;\; n \to \infty,
		\end{equation}
		where $ \mathrm{diam}(A) = \sup \{ \delta_j(\textbf{x}, \textbf{y}) : \textbf{x}, \textbf{y} \in A \} $, for a set $ A $ in $ \mathcal{F}_j.$ 		
		
		Determine  the  function $ \delta_j(A, B)= \inf \{ \delta_j(\textbf{x}, \textbf{y}) : \textbf{x}\in A, \, \textbf{y} \in B \},$   for  two nonempty bounded sets $ A $ and $ B $ in $ \mathcal{F}_j.$  The set  $ \mathcal{F}^j, j = 1,2,\ldots,$ satisfies the \textit{separation condition }of degree $ n ,$ if there exist a positive number $ \varepsilon^j_0 $ and a natural number $ n  = n(j)$  such that for arbitrary   indices  $ i_1 i_2 ... i_n $ one can find   indices $ j_1 j_2 ... j_n $ such  that
		\begin{equation} \label{C2}
		\delta_j \big( \mathcal{F}^j_{i_1 i_2 ... i_n} \, , \, \mathcal{F}^j_{j_1 j_2 ... j_n} \big) \geq \varepsilon^j_0.
		\end{equation}
		 If    the   diameter and separation conditions are valid,  then  $\mathcal {F}^j$   is   said to be  a \textit{chaotic  structure}  for   $F_j, j=1,2,\dots,$  and  a \textit{chaotic module}  for  $F.$
		 
		 We shall  say  that  the  union,  $\mathcal{F},$  of all  modules   $\mathcal{F}_j, j=1,2,\dots,$   is a \textit{modular chaotic  structure}  for the union $F$ of all  sets $F_j, j =1,2,\ldots,$   if   the diameter  condition is valid for all  $j=1,2,\ldots,$   and there  exists a  positive  separation constant   $\varepsilon_0  = \inf_{j=1,2,\ldots}  \varepsilon^j_0.$  		
			
			Let us introduce the maps $ \varphi_j : \mathcal{F}^j \to \mathcal{F}^{j+1}, j =1,2,\ldots,$ such that
		\begin{equation} \label{MapDefn}
		\varphi_j (\mathcal{F}^j_{i_1 i_2 ... i_n ... }) = \mathcal{F}^{j+1}_{i_2 i_3 ... i_n ... }.
		\end{equation}		
		In  what follows,  we refer to the  map  $\varphi$ on the  set  $\mathcal F,$  which  is specified by  equations (\ref{MapDefn}).    That  is, $\varphi(\mathcal {F}^j)= \varphi_j(\mathcal {F}^j), j=1,2,\ldots.$    Moreover,   accept  that  $\delta(\mathcal{F}^j_{i_1 i_2 ... i_n ... }, \mathcal{F}^j_{j_1 j_2 ... j_n ... })  = \delta_j (\mathcal{F}^j_{i_1 i_2 ... i_n ... }, \mathcal{F}^j_{j_1 j_2 ... j_n ... }),  j =1,2,\ldots.$  
				   
		   Considering iterations of the map, one can verify that
		   \begin{equation} \label{MapSubset}
		 (\varphi_{j+n}(\ldots(\varphi_j(\mathcal F^j_{i_1 i_2 ... i_n}))\ldots) = \varphi^n(\mathcal F^j_{i_1 i_2 ... i_n}))= \mathcal{F}^{j+n+1},
		   \end{equation}
		   for arbitrary natural number $ n $ and $ i_k=1,2, ..., m, \; k=1, 2, ... \, $. The relations (\ref{MapDefn}) and (\ref{MapSubset}) give us a reason to call $ \varphi$ an \textit{affine similarity map},   since the shifts for lower and upper indexes are common. One can  call  also the map   \textit{modular similarity map}.    Moreover,  the number $ n $  is  named the  \textit{order of similarity}.  Analogously to  the   research  in \cite{AkhmetSimilarity},  we will  call  the set    $\mathcal F$  \textit{modular  abstract self-similar set} .
		   		   
		     The set  $\mathcal F^j,$ for a fixed $j =1,2,\ldots,$ is an  \textit{abstract self-similar set} \cite{AkhmetSimilarity}.
		     If $\mathcal F^j$   does not  depend on $j,$   the  modular similarity  map $\varphi$    is known  as  the \textit{similarity  map}   \cite{AkhmetDiscreteRandom,AkhmetDomainStruct}.

		     		  It  is clear that  there  is a naturally  determined map  $\varPhi: F \to  F,$    which  	  is specified  as $\varPhi: F_j \to  F_{j+1}, j=1,2,\ldots,$   and values $\Phi(f)$  for  a fixed $f \in F_j$ constitute  the  set $\{\phi(\mathcal{F}^j_{i_1 i_2 ... i_n \ldots})\},$  with   all labels of the point $f.$   The  map  $\varPhi$   does not  satisfy  the  uniqueness condition. 
		     In what  follows, we shall describe chaotic  properties of  the  map  $\varPhi$  in terms of the  dynamics of the map  $\phi.$  
		   
		      We shall  say that  an  element  $\mathcal F^j_{i_1 i_2 ... i_n\ldots}$ is in $(k,m)-$neighborhood of a point  $\mathcal F^p_{s_1 s_2 ... s_n\ldots}, j < p,$  if  $\mathcal F^{j+k}_{i_{k+1} i_{k+2} ... i_{k+m}} = F^p_{s_1 s_2 ... s_m}.$ In other  words,  $j+k = p,  i_{k+1} = s_1,\ldots, i_{k+m}=s_m.$   	
		   			   		
		   		 Since the concepts  are based on the  Bernoulli shift \cite{Wiggins88},  known  for the symbolic dynamics,  it  is easily to  see that the following   chaotic properties are valid. 
		   			   
		   A point $ \mathcal{F}^j_{i_1 i_2 i_3 ...}, j=1,2,\ldots,$ is periodic with period $p,$ if its lower index consists of endless repetitions of a block of $ p$ terms.  The map $\varphi$   is said to  be \textit{modular-periodic}, 
		   since $\varphi^{kp}(\mathcal F^j_{i_1 i_2 ... i_n\ldots})= \mathcal{F}^{j+kp}_{i_1 i_2 ... i_n\ldots}$  for  any  natural $k,$  if the point   
		   $\mathcal{F}^j_{i_1 i_2 ... i_n  ... } \in  \mathcal{F}^j$   is  $p-$periodic.    
		   
		   Similarly,   the map  is \textit{modular-dense}  or \textit{modular-transitive}, since   there  exists  the  lower index $i_1 i_2 i_3 \ldots,$   such  that  for   any   sequence    $j_1 j_2 j_3 \ldots$  and natural $j$    and $l$  one can  find  positive  integers  $m$  and $k,  j < k,$  such  that  $\varphi^m(\mathcal F^j_{i_1 i_2 ... i_n\ldots})= \mathcal{F}^{k}_{j_1 j_2 ... j_n\ldots},$  and $i_s =j_s, s =1,\ldots,l.$  In other words, 
		   $ \mathcal F^j_{i_1 i_2 ... i_n\ldots}$  is in $(m,l)-$neighborhood of $ \mathcal{F}^{k}_{j_1 j_2 ... j_n\ldots}.$

		     Finally,    one can   say  that  the periodic points are   \textit{modular-dense} in $\mathcal F,$    since  for  any  point    $\mathcal{F}^{k}_{j_1 j_2 ... j_n\ldots}$  of the  set   and  natural numbers $l$  and $m  \le k$  one can  find  a periodic point  $\mathcal{F}^{m}_{i_1 i_2 ... i_n\ldots}$  such  that   it  is in $(k-m,l)-$   neighborhood of  $\mathcal{F}^{k}_{j_1 j_2 ... j_n\ldots}.$

     		    It   is clear   that couples $(\mathcal {F}^j, \delta_j),  j=1,2,\ldots,$  are not, in general,  metric spaces. Nevertheless,    the  map  $\varphi$   is convenient to    extend  all   attributes of  Poincar\'{e},  Li-Yorke and Devaney  chaos  \cite{AkhmetBook, Yorke,Devaney} for  the  dynamics,  since   the spaces $(F_j,d_j)$  are  metric.   We shall say  that  the map $\varphi$ is  \textit{modular  chaotic in the Devaney sense} if it  is  dense with  respect  to modular-periodic  trajectories, modular-transitive  and \textit{modular-sensitive}.  The  modular-sensitivity means,  that    there  exists a positive number $\varepsilon_0$   such  that  for  each  element $\mathcal F^j_{i_1 i_2 ... i_n\ldots}$  and arbitrary  positive $\kappa$  one can find  an element   $\mathcal F^j_{j_1 j_2 ... j_n\ldots}$   and  a natural  number  $k,$   which  satisfy $\delta(\varphi^k(\mathcal F^j_{i_1 i_2 ... i_n\ldots}), \varphi^k(\mathcal F^j_{j_1 j_2 ... j_n\ldots}) >\varepsilon_0,$  despite  $\delta(\mathcal F^j_{i_1 i_2 ... i_n\ldots}, \mathcal F^j_{j_1 j_2 ... j_n\ldots}) < \kappa.$  Finally,  the map $\varPhi$  is \textit{modular  chaotic  in the  sense of Devaney,}  if the  corresponding  map  $\varphi$ is modular  chaotic  in the same sense. 
     		    \begin{theorem} \label{Thm1} If $\mathcal F$ is a  modular  chaotic  structure, then  the dynamics of  $\varPhi$   is modular  chaotic in the sense of Devaney.
     		   \end{theorem}      
     		     \begin{proof}  	Modular-transitivity   and density  for modular-periodicity   are  valid, since of the  definitions.		   	
		   	For modular-sensitivity, fix a point $ \mathcal{F}^j_{i_1 i_2 ... } \in \mathcal{F} $ and an arbitrary positive number $ \kappa $. Due to the diameter   and separation conditions, there exist an integer $ k $ and element $ \mathcal{F}^j_{i_1 i_2 ... i_k j_{k+1} j_{k+2} ...}$ such that $ 0 < \delta(\mathcal{F}^j_{i_1 i_2 ... i_k i_{k+1} ...}, \mathcal{F}^j_{i_1 i_2 ... i_k j_{k+1} j_{k+2} ...}) < \kappa $ and  $\delta(\mathcal{F}^j_{i_{k+1} i_{k+2} ... i_{k+n}}, \mathcal{F}^j_{j_{k+1} j_{k+2} ... j_{k+n}}) > \varepsilon_0 .$  This proves the sensitivity.	
		   \end{proof}		   
		   In \cite{AkhmetUnpredictable,AkhmetPoincare}, Poisson stable motion is utilized to distinguish  chaotic behavior from  periodic motions in Devaney and Li-Yorke types.  The dynamics  is named  Poincar\`{e} chaos.    Let  us provide the definition adapted  for  the  dynamics of  the  present  research.   A point  $ \mathcal{F}^j_{i_1 i_2 ... } \in \mathcal{F}, j=1,2,\ldots, $  is said to  be the  \textit{modular unpredictable point}    for the  map  $\varphi$,  if there    exist a positive number $\varepsilon_0$  and  two    unbounded sequences of natural  numbers,  $\kappa_n$  and $\zeta_n, n =1,2,\ldots,$  such  that  for  arbitrary    natural number   $l$  one can  find sufficiently  large  $\kappa_n$  such  that 
		   $\varphi^{\kappa_n}( \mathcal{F}^j_{i_1 i_2 ... }) =   \mathcal{F}^{j+\kappa_n}_{j_1 j_2 ... },$  where  $j_s = i_s,  s =1,\ldots,l,$  and  
		   $j_{\kappa_n+\zeta_n} \not = i_{\kappa_n+\zeta_n} .$ 		   
		   
		     Following the definition  of Poincar\'{e} chaos \cite{AkhmetPoincare} we shall  say  that  the   dynamics of the map $\varphi$ is \textit{modular-chaotic in the sense of Poincar\'{e}} if there  is  a modular-unpredictable point  for each  $j=1,2,\ldots.$     Moreover,   we call  the unpredictable point  itself  an  unpredictable orbit  of the dynamics.  It  is natural to  accept  that the  map  $\varPhi$  is  \textit{modular  chaotic  in the  sense of Poincar\'{e},}  if the  corresponding  map  $\varphi$ is modular  chaotic  in the   sense of Poincar\'{e}. 
		    \begin{theorem} \label{Thm2}  If $\mathcal F$ is a   modular  chaotic  structure, then the map  $\varPhi$  possesses modular chaos in the sense of   Poincar\`{e}.
		   \end{theorem}		   
		   The proof of the last theorem is based on the verification of Lemma 3.1 in \cite{AkhmetPoincare},  applied  to the modular similarity map.
		   
		   In addition to  modular Devaney and Poincar\`{e} chaos,  the chaos of the Li-Yorke  type  \cite{Yorke} can  be defined.   Additionally  to  the  modular-periodic points  of the dynamics,  assume that  there  exists  an uncountable \textit{scrambled set} set  of  non-periodic points in   $ \mathcal{F}^j, j =1,2,\ldots,$   such  that 
		   		   $$(\alpha) \lim\sup_{k \to \infty}\delta( \varphi^k(\mathcal{F}^j_{i_1 i_2 ... }), \varphi^k(\mathcal{F}^j_{i_j j_2 ... } ) )  >0,$$  
		   		     and 
		   		   $$(\beta) \lim\inf_{k \to \infty}\delta( \varphi^k(\mathcal{F}^j_{i_1 i_2 ... }), \varphi^k(\mathcal{F}^j_{i_j j_2 ... } ) ) =0,$$
		   		   for each  couple  in  the set,  if   $\delta( \mathcal{F}^j_{i_1 i_2 ... },   \mathcal{F}^j_{i_j j_2 ... } ) \not  = 0.$   And, finally,  assume that    for each    point  $     \mathcal{F}^j_{i_1 i_2 ... }$  of the  set  and a periodic  point  $\mathcal{F}^j_{i_j j_2 ... }$  it  is true that   $\lim\sup_{k \to \infty}\delta( \varphi^k(\mathcal{F}^j_{i_1 i_2 ... }), \varphi^k(\mathcal{F}^j_{i_j j_2 ... } ) )  >0.$   The  conditions  $(\alpha)$  and $(\beta)$  are called 
		   		   \textit{modular-proximality}  and \textit{modular-frequent  separation} properties  for \textit{the modular  Li-Yorke chaos},  respectively.

		 	We will  say   that for the sets $ \mathcal{F}^j_{i_1 i_2 ... i_n}, j=1,\ldots,$ the \textit{strong  diameter condition} is valid,  if 
		 \begin{equation} \label{Diamprop}
		 \mathrm{sup}_{j} \mathrm{max}_{i_1 i_2 ... i_n}\mathrm{diam}(\mathcal{F}^j_{i_1 i_2 ... i_n}) \to 0 \;\; \text{as} \;\; n \to \infty.
		 \end{equation}		 
		  	 We shall  say  that  the  union,  $\mathcal{F},$  of all  modules   $\mathcal{F}_j, j=1,2,\dots,$   is a \textit{strong modular chaotic  structure}  for the union $F$ of all  sets $F_j, j =1,2,\ldots,$   if   the strong diameter  condition is valid for all  $j=1,2,\ldots,$   and there  exists a  positive  separation constant   $\varepsilon_0  = \inf_{j=1,2,\ldots}  \varepsilon^j_0.$    We shall  call  the  map  $\varPhi$   \textit{modular  chaotic  in the  sense of Li-Yorke,}  if the  corresponding  map  $\varphi$ is modular  chaotic  in the sense of Li-Yorke. 
		     \begin{theorem} \label{Thm3}  If $\mathcal F$ is a   strong modular  chaotic  structure, then the map  $\varPhi$   is modular chaotic  in the  Li-Yorke sense.
		  \end{theorem}		      
		      The proof of the  assertion is similar to that of Theorem 6.35 in \cite{Chen} for the shift map defined in  the space of symbolic sequences.
		   
		   One can  easily  see  that,  in the  case when $(F_j,d_j)$ does not  depend on $j,$ we obtain   domain-structured chaos,  which  was introduced  in \cite{AkhmetDiscreteRandom}.  That is,   modular chaos can be considered as generalization of  domain-structured chaos.     The  modular chaos relates to domain structured chaos  similarly  as non-autonomous dynamics does   to  autonomous one. 
		   
		   The  analysis  of   ingredients for  the  three types of chaos  implies   that   the most weak,  and, consequently,  most  flexible properties are  required   for  Poincar\'{e} modular   chaos,   since there  distance evaluations   are  not  request. Nevertheless,  we are confident  that   this must  not    be considered as a  deficiency.   The description still  can  provide circumstances, sufficient  for  irregularity,   which  can  be useful   for   future   applications. 
		   \section{Chaotic random processes}
	 Let  us describe the  type of the  stochastic  processes \cite{Doob} that    are   in the focus  of this  paper.   Consider a continuous  time or discrete time  random    process  with continuous or discrete state space  as a family of random variables $\textbf X(t)$   indexed by  the parameter  $t$  with  the range  $\mathcal I,$  which  is either an interval  of   the real line  or  an   infinite  set  of integers bounded from below.   Realizations of the dynamics are  not  necessarily  continuous  if the time is continuous  and  they  must  not  be  bounded functions for both sorts of time. Denote  by 
	$S = \bigcup_{t\in \mathcal I} S^t$ the  state  space of the process,  and  consider it  with  a distance $d.$  We assume positive probability  for all  members   of the state  space,  and for  each  fix  $t$   one of the  members of the state space must   necessarily happen.     

	Suppose that    for   each   fixed $t \in \mathcal I$  there exists   a   presentation, 
		\begin{equation} \label{DomainStrSpace}
	\mathcal{S}^t =  \big\{\mathcal{S}^t_{i_1 i_2 ... i_n ... } : i_k=1,2, ..., m, \; k=1, 2, ... \big\},
	\end{equation}	     		     	
	where    $m$  is a fixed natural  number, common for all $t \in \mathcal I,$  larger  or  equal  than  two,   such  that  each   element of the  set  $S^t$ is presented  as a   member   of the  set  $\mathcal S^t.$  The  presentation is in the sense of the last section, 	 and corresponding  to  the  distance $d$ function $\delta$  is common for  all $t \in \mathcal I.$
	
	Consider all  possible   infinite strictly  increasing   sequences $t_{\alpha} = \{t_i^{\alpha}\} ,  i = 1,2,\ldots,$  in $\mathcal I,$  with  $\mathcal A,$   the  set  of  all  indexes  $\alpha.$   
	
	We shall say  that  the  random process $\textbf X(t)$  is \textit{(strongly) modular chaotic},   if $\mathcal S^{\alpha} = \bigcup_i \mathcal S^{t_i^{\alpha} }$  is a  (strong)  modular chaotic structure  for $S^{\alpha} = \bigcup_i S^{t_i^{\alpha} }$    with  each  $\alpha \in \mathcal A,$    and   there  exists  a number $\varepsilon_0,$  which  is larger than all  separation constants $\varepsilon^{\alpha}_0,  \alpha \in \mathcal A.$  	
		
				 The   choice of modular  chaos for the random dynamics description  is  approved,  since  the  set  of all  realizations   for  dynamics  $\textbf X(t_{\alpha})$  with  fixed $\alpha \in \mathcal A$  coincides with  the  set  of all  trajectories of  the  affine  similarity  map $\varphi$  on the set    $\mathcal S^{\alpha},$  and  the  dynamics of the  affine  map  is similar to  the   symbolic dynamics.   To argument this,  let us turn the  discussion to  the lower indexes  analysis.  Consider the space $\displaystyle \Sigma_m=\big\{(i_1i_2\ldots) \ |  i_k =1,2,\ldots,m \big\}$ of infinite sequences on  finite number of symbols with the metric $$d(s,t)=\displaystyle \sum_{k=1}^{\infty} \frac{\left|i_k-j_k\right|}{2^{k-1}},$$ where $(i_1i_2\ldots),$ and  $(j_1j_2\ldots)$   are  elements of  $ \Sigma_m.$ The Bernoulli shift $\sigma: \Sigma_m\to \Sigma_m$ is defined as $\sigma(s_0s_1s_2\ldots)=(s_1s_2s_3\ldots).$ The map $\sigma$ is continuous and $\Sigma_r$ is a compact metric space \cite{Wiggins88,Devaney}. According to the  result  in \cite{AkhmetPoincare}, the   symbolic dynamics  admits  an unpredictable 􏰁  point, $i^*,$   a sequence  from the  set \cite{AkhmetUnpredictable}.   It  is important for  us that  $\Sigma_m$  is a closure  for  the unpredictable orbit.   The peculiarity  of the  metric   implies   that    each   arc of any  sequence in  the space coincides with some arc of the  unpredictable  sequence with   a   shift  precision.   In other words,   since of finite  number of indexes  involved in the  construction,  one  have the  complete coincidence  with   a part  of the unpredictable orbit. Moreover,   the  coincidence  does not  last  forever,   since of the unpredictability.    This is why,   each  iteration of the  Bernoulli shift  $\sigma$  results  what  is done by  the Bernoulli  trial and,  similarly, by  $\textbf X(t_{\alpha}).$  This  is why,   we  accept that  the  random dynamics  is modular  chaotic.     
				 
				 It  is clear  that  for  the discrete  time random process, it is sufficient to  consider  only   the sequence related to  the  set  of indexes $\mathcal I$  itself.  We do  not   exclude  that the  method  of modular  chaos  may  be of use even  if the  modular  chaotic structure  exists not    for  all  $\alpha$    from $\mathcal A,$  but  for some of them. 
				 
				 On the  basis of the  last  discussion one  can  say that  stochastic processes  are  more "saturated with  irregularity",  than  deterministic chaos.   Nevertheless,  we have the strong confidence that   through  the  developed  results   one  can  make the  difference  between  the  two  concepts  less,  if not  diminish  at  all. 
				 Next, we will  provide three  examples of random processes,  which  are  modular  chaotic.
		           \begin{example}
		           Consider  several scalar real-valued  functions $f_j(t), j = 1,\ldots,m,$ defined on a  time  interval, $\mathcal I,$  of real  numbers.  Assume that  the  ranges of the functions are  disjoint  and   minimal  distance between  elements of  distinct  ranges is larger  than a positive  number $\varepsilon_0.$   Define the  random process $\textbf X(t)$   such  that    for  each  fixed $t \in \mathcal I,$  it  admits  value of one of the  functions $f_j$  with positive probability  $p_j,  j =1,\ldots,m.$  The probabilities do not  depend on $t,$  and  the  sum  of  all  probabilities  is  equal  to the unit   if   $t$  is fixed.   
		           
		           We have that   the  state space is $S = \bigcup_{t \in \mathcal I} S^t,$ where  $S^t = \{f_j(t)| j=1,\ldots,m \}.$   Let  us construct  the presentation  
		           \begin{equation} \label{DomainStrSpace}
		           \mathcal{S}^t =  \big\{\mathcal{S}^t_{i_1 i_2 ... i_n ... } : i_k=1,2, ..., m, \; k=1, 2, ... \big\},
		           \end{equation}
		           where   $\mathcal{S}^t_{i_1 i_2 ... i_n ... } = f_{i_1}(t), i_j = 1,2,\ldots,m,\; j =2,\ldots.$  One can   check  that  the  union, $\mathcal S,$  of sets  $\mathcal S^t, t  \in \mathcal I,$  is a modular chaotic  structure  for  the set  $S,$  and   consequently  the  	random process is modular  chaotic. 
	           \end{example}  
           \begin{figure}[ht]
           \centering
           \includegraphics[height=7.2cm]{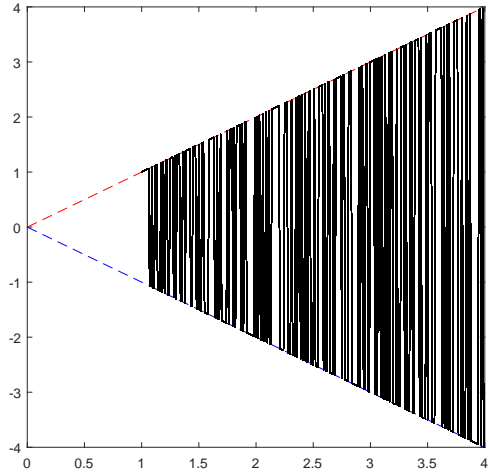}
           \caption{The graph of the function $\varphi$  in interval $[1,4]$. The red dashed line is the graph of the identity function $t$ and the blue dashed line is the graph of function $-t$.}
           \label{p1-2}
       \end{figure}
   \begin{example}  (numerical simulation)  	
   Let  us  consider   the  random process  	$\textbf  X(t),$     which  is equal  to  $t$  or $-t$    for  each  fixed $t  \in [1,4],$  with  probability  $1/2.$  It  is clear  that  the dynamics is  a continuous-time random process with  continuous   state  space.  Fix the parameter value $t$  and obtain that  $S^t = \{t,-t\}$ consists  of two  elements,  and the distance between  the   two  points  is equal  $|t - (-t)| = 2t  \ge \varepsilon_0 = 2.$  Let  us,  construct   a chaotic structure  for  the  set  $S^t.$   Denote $S^t_{i_1i_2\ldots} =t $  with  $i_1 = 1, $  and   $i_j$ is equal  to   $1$  or $2$  for  all $j  \ge 2.$  Similarly,         $S^t_{i_1i_2\ldots} = - t$  with $ i_1 = 2$  and $i_j$  is equal  to $1$  or $2$  for  all $j  \ge 2.$    According  to   discussion  in the last example the random process   is modular  chaotic. To illustrate  the dynamics,
   we consider the  sequence $t_i = i/100,  i = 100,2,\ldots,400,$    randomly   determine values    $\textbf  X(t_i)$    and   draw the  graph of the   piece-wise constant function  $\varphi$ equal to  $\textbf  X(t_i)$ on the  interval   $\big [i/100, (i+1)/100\big),  i =100,\ldots,399.$  The graph is seen  in Figure \ref{p1-2} and it  approximates one of realizations of the random dynamics.    	
\end{example}  
\begin{figure}[ht]
\centering
\includegraphics[height=7.2cm]{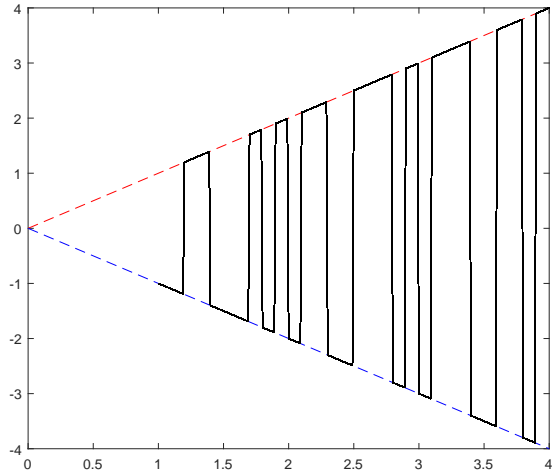}
\caption{The graph of the function $\psi$  in interval $[1,4]$. The red dashed line is the graph of the identity  function $t$ and the blue dashed line is the graph of function $-t$.}
\label{p3}
\end{figure}
\begin{example}  (discrete  time  dynamics)  Let  us 		determine   functions $f_i(t)=t$  and $g_i(t)=-t$  on the  interval  $\big [i/10, (i+1)/10\big)$  for  each $i=10,11,\ldots.$     Construct  the metric space  of continuous functions, $S^i = \{f_i(t),g_i(t)\}, t \in  \big [i/10, (i+1)/10\big),$   for  each  $i=10,11,\ldots,$  with  the  distance $d(f_i(t),g_i(t))= \max_{t} |t - (-t)|.$   Introduce the discrete  time random process $\textbf X(i),  i=10,11,\ldots,$    which   is equal  to  the  function $f_i(t)$ or $g_i(t)$  with  probability $1/2$  on the interval  $\big [i/10, (i+1)/10\big).$  One can   easily  see that  the    random dynamics is   modular  chaotic,   since one can construct  appropriate   modular   chaotic  structure.  The graph  of  a    realization, $\psi,$   on the  interval $[1,4]$ is seen  in Figure \ref{p3}.
\end{example}

		   	\section{Conclusion} 
		   		We  have pursued   three   goals  in    the  present   research. The first  one  is to reveal  deterministic chaos  signs   in  random phenomena,   and the second,   to maximize the presence.   New   definitions of chaos to  fulfill  the  two tasks   have been  developed.  All  the  three are  considered in   this  research   for  random processes with discrete and continuous time. 		   	
		   	   	The application of the  modular  chaos concept   to   random dynamics  analysis  is  sensible, since  the  set  of all  realizations   of    a random  dynamics   coincides with  the  set  of all  trajectories of the chaos.  	To approve the phenomenon,   the domain structured chaos notion, which  has been considered in previous our  papers,  is   extended.  		   	   	
		   	The  new type of   chaos   has been  developed   such  that  the  notions of transitivity and density  of periodic  orbits  are free of metrical or topological features, and are  based on the comprehension that  two  points are  close, if they  belong to  the  same set by indexing.    Moreover,  trajectories  move among different sets, when  the  time increases,    and  each  of the  sets  admits its own chaotic structure. It  is important   that  the concept    helped us to  find  a chaotic structure  for  continuous-time random processes, bounded  as well as  unbounded.  
		   	We accept  the phenomenon  as  a sign of  deterministic chaos in the dynamics,  since  realizations of the corresponding  discrete  time random process   are the same as trajectories of  the  deterministic  modular  similarity  map introduced in the present  research. 
		   	The  concept  can  be   considered for  processes,  which  are stochastic only  partially, in  time,  that  is hybrid dynamics, through combinations of intervals,  where  a dynamics is deterministic  or  stochastic.   	   		   	
		   	  Hopefully,  this  all  may generates new  approaches inside of the stochastic dynamics theory \cite{Oksendal}  as well  as   for non-stationary processes  and  for  Markov chains.   Our research  results  may  help  for development of chaos in non-equilibrium  processes \cite{Prigogine}.
		   	The methods of    chaotic dynamics, e.g. synchronization and control   \cite{Gonzales}, can be extended to dynamics with probability.   Moreover, one can expect  that results  for symbolic dynamics 
		   	\cite{Williams} considered in the course of the  present  research  can give new  effects  for entropy \cite{Zmeskal,Adler},  harmonic analysis, discrete mathematics, probability, and operator algebras \cite{Jorgensen}.

	   	\end{document}